\documentclass[11pt]{amsart}
\usepackage{amsfonts,amsmath,amssymb,fullpage,graphicx, comment}
\usepackage{verbatim}
\usepackage{xcolor}
\newtheorem{thm}{Theorem}[section]
\newtheorem{cor}[thm]{Corollary}
\newtheorem{lem}[thm]{Lemma}

\theoremstyle{definition}
\newtheorem{defn}[thm]{Definition}
\theoremstyle{remark}
\newtheorem{rem}[thm]{\bf{Remark}}
\numberwithin{equation}{section}

\newcommand{\beas}{\begin{eqnarray*}}
	\newcommand{\eeas}{\end{eqnarray*}}
\newcommand{\bes} {\begin{equation*}}
	\newcommand{\ees} {\end{equation*}}
\newcommand{\be} {\begin{equation}}
	\newcommand{\ee} {\end{equation}}
\newcommand{\bea} {\begin{eqnarray}}
	\newcommand{\eea} {\end{eqnarray}}
\newcommand{\ra} {\rightarrow}
\newcommand{\txt} {\textmd}

\newcommand{\R}{\mathbb R}

\newcommand{\C}{\mathbb C}

\newcommand{\N}{\mathbb N}

\begin{document}
	
	\title[Uniqueness results for quasi-analytic functions on compact Lie groups] {Uniqueness results for quasi-analytic functions on compact Lie groups and homogeneous spaces}
	
	\author{Mithun Bhowmik and Sanjib Pradhan}
	
	\address{(Mithun Bhowmik) Department of Mathematics, Indian Institute of Technology Kharagpur-721302, India}
	\email{mithun@maths.iitkgp.ac.in}
	
	\address{(Sanjib Pradhan) Department of Mathematics, Indian Institute of Technology Kharagpur-721302, India}
	\email{sanjibpradhan01.24@kgpian.iitkgp.ac.in}

	\thanks{The second author was supported by Junior Reseach Fellowship from IIT Kharagpur,  India}
	
	
\begin{abstract}
In this article, we establish a quantitative uniqueness theorem for quasi-analytic functions defined on compact, connected Lie groups $G$ and on homogeneous spaces $G/H$,  where $H$ is any closed subgroup of $G$.  Our result extends classical  Logvinenko-Sereda-type theorems to the setting of quasi-analytic functions on compact Lie groups and their homogeneous spaces.
		
		We introduce the quasi-analytic class of functions using iterates of the Casimir operator on $G$. This construction is justified by establishing that every function in this class possesses the strong unique continuation property.  In particular,  our result extends a result of P. Chernoff (Bull. Amer. Math. Soc., 1975) to the framework of compact Lie groups and their homogeneous spaces.
	\end{abstract}
	
	
	
	\maketitle
	\section{Introduction}
	
	In this paper, we address the quantitative uniqueness problem for quasi-analytic functions defined on compact, connected Lie groups $G$ and their homogeneous spaces $G/H$.  Specifically,  we provide a uniform control on the $L^p$-norm of quasi-analytic functions on the whole space by the $L^p$-norm of their restriction on a relatively dense set.  The quasi-analytic class is defined in terms of iterates of the Casimir operator $\mathcal L$ on $G$.  We show that any function in this class satisfies the strong unique continuation property.   
	
\subsection{Quasi-analytic class on $G$}
We begin with the definition of quasi-analytic class of functions on $\R^d$.
\begin{defn}
A class $S$ of $C^\infty$ functions on a domain $\Omega \subseteq \R^d$ is said to be a quasi-analytic class if whenever a function $f$ in $S$,  and all its partial derivatives vanish at a point in $\Omega$,  then $f \equiv 0$ on $\Omega$. 
\end{defn} 
	For a real analytic function $f$ defined on an interval $(a, b)\subset \R$,  it follows from the formula of the remainder term associated with the Taylor series that the derivatives of $f$ cannot grow too fast.  This motivates one to find appropriate growth conditions on $C^\infty$ functions to define a class which accommodates functions which are not necessarily analytic,  yet a nonzero member cannot vanish at a point along with all its derivatives.  To define the class of quasi-analytic functions on an interval $I\subseteq \R$,  we introduce the following notations. A sequence $\mathcal M = \{M_n\}_{n=0}^\infty$ of positive numbers is said to be log-convex if
	\bes
	M_n^2 \leq M_{n-1} M_{n+1}, \:\: \textit{  for all } n\in \N.
	\ees
	We consider the following class of smooth functions defined on $I$ associated with $\mathcal M$,
	\be \label{defn-C}
	\widetilde C_{\mathcal M}(I)=\left\{f\in C^\infty(I): \|f^{(n)}\|_{L^\infty (I)} \leq B_f \beta_f^n M_n, \:\: \forall n \in \N_0 \right\}.
	\ee
	Here,  $f^{(n)}$ denotes the $n$-th derivative of $f$ and $\N_0=\N \cup \{0\}$ is the set of non-negative integers.  The positive constants $B_f, \beta_f$ depend on $f$ but not on $n$.   A necessary and sufficient condition on the log-convex sequence $\mathcal M$ to generate a quasi-analytic class is given by the following seminal result of Denjoy and Carleman.  
	\begin{thm} \cite[Theorem 19.11]{Ru}  \label{thm-DC}
		The class $\widetilde C_{\mathcal M} (I)$ is quasi-analytic if and only if
		\be \label{quasi-cond}
		\sum_{n\in \N} \frac{M_{n-1}}{M_n} =\infty, \:\: \textit{ or equivalently, } \:\: \sum_{n\in \N} M_n^{-\frac{1}{n}}=\infty.
		\ee
	\end{thm}
	
	In view of above theorem,  we call a log-convex sequence $\mathcal M$ quasi-analytic if the associated class of smooth functions $\widetilde C_{\mathcal M}(I)$ is quasi-analytic,  i.e.  if (\ref{quasi-cond}) holds. Since the quasi-analyticity condition restricts how fast the derivatives of a function can grow, we may assume that $M_n$ is non-decreasing.  Later,  Bochner and Taylor extended Theorem \ref{thm-DC} to $\R^d$ and other spaces \cite{B,  BT}, in which various special differential operators were constructed to replace $d/dx$ used in the definition (\ref{defn-C}).   Some important developments in this direction are the study of analytic vectors of elliptic operators on Lie groups by Nelson \cite{N} and subsequent work on quasi-analytic vectors by Nussbaum  \cite{Nu} in the set-up of operator theory on Hilbert spaces.  In this article,  we will be concerned with the following result of Chernoff,  which he obtained as part of his study on the relationship between operator-theoretic developments on quasi-analytic vectors and quasi-analytic functions.   Let $\Delta_{\R^d}$ be the Laplacian on $\R^d$. 
	\begin{thm} \cite[Theorem 6.1]{Ch}, \label{thm-cher}
		Let $f:\R^d\to\C$ be a smooth function such that
		for all $n\in \N_0$,  $\Delta_{\R^d}^n f\in L^2(\R^d)$, and
		\be \label{Carl-cond-Rn}
		\sum_{n\in \N}\|\Delta_{\R^d}^n f\|_{L^2(\R^d)}^{-\frac{1}{2n}}=\infty.
		\ee
		If there exists $x$ in $\R^d$, such that $f$ and all its partial derivatives $\partial^{\alpha}f=\frac{\partial^{\alpha_1+\ldots +\alpha_d}f}{\partial x_1^{\alpha_1}\ldots\partial x_d^{\alpha_d}}$, $\alpha\in \N_0^d$,  vanish at $x$ then $f$ vanishes identically on $\R^d$.
	\end{thm}
	The above result differs from earlier works in several respects.  In particular,  in contrast with (\ref {defn-C}),  the growth estimate now involves only powers of the Laplacian,  and the $L^2$-norm replaces $L^\infty$-norm.   Theorem \ref{thm-cher} was previously extended to Riemannian symmetric spaces in joint work of the first author with S. Pusti and S. K. Ray \cite{BPR1, BPR},  under certain symmetry assumptions on the function class.  Subsequently,  the result has been extended in generality to symmetric spaces of noncompact type by R. P. Sarkar \cite{Sarkar} without imposing such symmetry assumptions; see also \cite{GMT} for the rank-one case.   We record the following result in the context of Riemannian symmetric spaces $G/K$ of compact type for our future use.   Here,  $\Delta$ denotes the Laplace-Beltrami operator on $G/K$, and we let ${\mathbb D}(G/K)$ denote the algebra of $G$-invariant differential operators on $G/K$.   
	\begin{thm} \cite[Theorem 1.5]{BPR1}\label{thm-cher-com}
		Let $f\in C^\infty(G/K)$ be a left $K$-invariant function which satisfies the condition
		\be \label{Carl-cond-com}
		\sum_{m\in \mathbb N} \|\Delta^m f\|_{L^p(G/K)}^{-\frac{1}{2m}} =\infty,
		\ee
		for some $p\in [1,\infty]$. If $Df$ vanishes at the identity coset $o$ for all $D\in {\mathbb D}(G/K)$, then f vanishes identically.
	\end{thm}
In this paper, we introduce a notion of quasi-analytic function classes on compact Lie groups,  extending earlier results established for particular cases.  Let $G$ be a compact, connected Lie group with Lie algebra $\mathfrak g$.  We note that the Casimir operator $\mathcal L$ on  $G$ serves as an analogue of the Euclidean Laplacian.  The group $G$ is equipped with the normalized Haar measure $\mu$,  so that $\mu(G)=1$.  For $p\in [1, \infty]$,  we consider the following class of smooth functions on $G$ associated with a quasi-analytic sequence $\mathcal M=\{M_n\}_{n=0}^\infty$ 
	\bes
	\widetilde C_{\mathcal M, p}(G)=\left\{f\in C^\infty(G): \|\mathcal L^nf\|_{L^p(G)} \leq B_f \beta_f^n M_{2n},  \:\: \forall n \in \N_0 \right\}.
	\ees
	As before, the positive constants $B_f, \beta_f$ depend on $f$ but not on $n$.  Every element $D$ in the universal enveloping algebra $\mathcal U(\mathfrak g)$ acts on $f\in C^\infty(G)$ as a left- and right-invariant differential operator $\widetilde D$ and $d_r(D)$ respectively; see (\ref{defn-dr}) for the precise definition.  In this article,  we prove the following strong unique continuation result on $G$.
	\begin{thm} \label{thm-cher-G}
		Suppose $f\in \widetilde C_{\mathcal M, p}(G)$ for some $p\in [1, \infty]$.  If there exists $g$ in $G$ such that $\widetilde D f$ vanishes at $g$,  for all $D\in \mathcal U(\mathfrak g)$,  then $f$ is identically equal to zero.  
	\end{thm}
	
	\begin{rem}
		\begin{enumerate}
\item[(i)] Theorem \ref{thm-cher-G} remains valid if $\widetilde D$ is replaced by $d_r(D)$.		
		
			\item[(ii)] In \cite[Sec. 7.2]{Jeu},  Marcel de Jeu proposes an interpretation of the quasi-analyticity on Lie groups.  Precisely,  he asks what conditions on a smooth function $f$ on a connected Lie group $G$ are sufficient to ensure that $f$ vanishes identically if $Df (e) = 0$ for all $D$ in the universal enveloping algebra of $G$ ? Theorem \ref{thm-cher-G} answers this question for compact,  connected Lie groups. 
			
			\item[(iii)] An elliptic regularity theorem due to T.  Koteke and M.S. Narashimhan \cite {KN} shows that  given an open set $\Omega \subset \R^d$ and a linear elliptic partial differential operator $P$ of order $m$ with analytic coefficients in $\Omega\subset \R^d$,  every smooth function $u$ which satisfies
			\bes
			\|P^k u\|_{L^2(\Omega)} \leq (km)! c^{k+1},  
			\ees
			for every  nonnegative integer $k$,  with $c$ independent of $k$,  is analytic in $\Omega$.  This theorem has been generalized in various directions by several authors. Theorem \ref{thm-cher-G} may be regarded as an analogue of this result for quasi-analytic functions on compact Lie groups.
\end{enumerate}
\end{rem}
	
	We now turn to the homogeneous spaces $G/H$, where $H$ is any closed subgroup of $G$.  Let $\mathfrak h$ denote the Lie algebra of $H$ and $\mathfrak q$ be its orthogonal complement in $\mathfrak g$ with respect to the $Ad$-invariant inner product on $\mathfrak g$.  The restriction of this inner product to $\mathfrak q$ induces a $G$-invariant Riemannian metric on $G/H$.  Denote by 
$\widetilde{\mathcal L}$ the corresponding Laplace-Beltrami operator on $G/H$.  Let 
$\pi: G \ra G/H$ be the canonical projection, and let 
 $\widetilde \mu$ be the pushforward of the normalized Haar measure $\mu$ on $G$ under $\pi$;  explicitly,  $\widetilde{\mu}(A)=\mu(\pi^{-1}(A))$ for a measurable subset $A\subseteq G/H$.  Given a quasi-analytic sequence $\mathcal M=\{M_n\}_{n=0}^\infty$,  we define the quasi-analytic class on $G/H$ associated with $\mathcal M$ as follows:
	\bes
	\widetilde C_{\mathcal M, p}(G/H)=\left\{f\in C^\infty(G/H): \| \widetilde{\mathcal L}^n f\|_{L^p(G/H)} \leq B_f \beta_f^n M_{2n},  \:\: \forall n \in \N_0 \right\}.
	\ees
 When $D\in \mathcal U(\mathfrak q)$,  the right-invariant operator $d_r(D)$ descends naturally to a differential operator on $G/H$.  We recall that any function $f$ on $G/H$ can be identified with a right $H$-invariant function on $G$.   In particular,  functions in the class $\widetilde C_{\mathcal M, p}(G/H)$ possess an additional symmetry compared to those in $\widetilde C_{\mathcal M, p}(G)$.  This suggests that,  to establish a unique continuation result on $G/H$,  it may not be necessary to impose the vanishing condition for every $D \in \mathcal U(\mathfrak g)$.  The following theorem confirms this expectation.  
	\begin{thm} \label{thm-cher-GH}
		Suppose $f\in \widetilde C_{\mathcal M, p}(G/H)$ for some $p\in [1, \infty]$.  If $d_r(D)f$ vanishes at $o=eH$ for all $D \in \mathcal U(\mathfrak q)$,  then $f$ is identically equal to zero.  
	\end{thm}

\begin{rem}
\begin{enumerate}
\item[(i)] Let us specialize to the case where the homogeneous space is a Riemannian symmetric space of compact type $G/K$. To the best of our knowledge,  the strongest existing result on quasi-analytic functions on $G/K$ of arbitrary rank is Theorem \ref{thm-cher-com}, which applies solely to $K$-biinvariant functions.  In that setting,  the vanishing condition requires that $Df$ vanish at the identity coset  for all $G$-invariant differential operators on $G/K$.  In contrast, Theorem \ref{thm-cher-GH} above extends Theorem \ref{thm-cher-com} to arbitrary quasi-analytic functions on $G/K$,  without assuming $K$-biinvariance.  In this more general framework,  the vanishing condition involves all differential operators arising from $\mathcal U(\mathfrak q)$,  rather than being restricted to 
$G$-invariant ones.  
			
\item[(ii)] In the case of the $d$-dimensional sphere $S^{d-1}$,   an analogue of Theorem \ref{thm-cher-GH} was established in \cite[Theorem 1.9]{GMT}. 
\end{enumerate}
		
\end{rem}
	
	\subsection{Uniqueness results for quasi-analytic functions on $G$}
	We now turn to the second central theme of this paper: establishing quantitative uniqueness properties on compact Lie groups and their homogeneous spaces. A classical example in this direction is provided by the Logvinenko-Sereda theorem \cite{LS, P}, which represents a prototypical form of the uncertainty principle in harmonic analysis \cite{CCR, FS, T}. To present these results, we start with the following definition.
	
	\begin{defn}
		Let $\gamma, l$ be two positive numbers.  A Borel set $E\subset \R^d$ is $(\gamma, l)$- relatively dense if $|E\cap Q| \geq \gamma l^d$ for any cube $Q \subset \R^d$ of side-length $l$. Here, $|A|$ denotes the Lebesgue measure on $\R^d$ for any measurable $A\subset \R^d, d \geq 1$.
	\end{defn} 
	
	\begin{thm}[Paneah-Logvinenko-Sereda Theorem \cite{LS, P}] \label{thm-PLS}
		Fix $E\subset \R^d$. For every $N > 0$, there is a constant $C > 0$ such that
		\bes
		\|f\|_{L^2(\R^d)}\leq C \|f\|_{L^2(E)},\:\: \textit{ for every } f \textit{ with } supp (\mathcal F f) \subset B(0, N),
		\ees
		if and only if $E$ is $(\gamma, l)$-relatively dense for some $\gamma \in (0, 1)$ and $l>0$.  
	\end{thm}
	Here, $\mathcal F f$ denotes the Fourier transform of a function on $\R^d$.  Later,  using a different method, Kovrijkine refined the bounds established by Logvinenko-Sereda and achieved the optimal constant-first in the one-dimensional case \cite{K}, and later in his thesis for the multi-dimensional setting. In particular, the constant $C$ can be taken as
	$C= \left(c^d/\gamma \right)^{c d( l N+1)}$, for some universal constant $c>0$.  The key ingredient for the proof of Theorem \ref{thm-PLS},  as presented in \cite{K},  is the reduction to a Remez-type inequality for analytic functions via the localisation principle.
	
	Because of their applications in areas of applied analysis,  Theorem \ref{thm-PLS} has been translated into a Hilbert space setting. Let $M$ be a non-empty set,  $H$ a self-adjoint operator on $L^2(M),$ and let $P_H((-\infty,E]) := \chi_{(-\infty,E]}(H)$ be the associated spectral projection up to energy $E \in \R$.  We say that $H$ satisfies a spectral inequality from some measurable set $S \subset M$ if for all energies $E \geq 0$,  there is a constant $C > 0$,  depending only on $M, S, H,$ and $E$, such that
	\bes
	\|f\|_{L^2(M)} \leq C\|f\|_{L^2(S)}, \:\: \text{ for all } \:\: f \in Ran~P_H((-\infty, E]).
	\ees
	The spectral inequality for the Laplace-Beltrami operator on the $d$-dimensional sphere is established in \cite{DV}, leading to observability and null-controllability results for the spherical heat equation  with explicit estimates on the associated control costs.  In \cite{EV}, using techniques developed  in \cite{K},  an analogous inequality was derived for the $d$-dimensional torus. The original motivation for this development comes from the aim of formulating sharp uncertainty principles for spectral projections of Schr\"odinger operators,  particularly in the context of random Schr\"odinger operators (see,  for instance, \cite{RV} and the references therein).  For further studies on spectral inequalities, we refer the reader to \cite{DSV, EV, OP, R}.

	Our point of departure from these results is a uniqueness theorem of Jaye and Mitkovski \cite{JM}, who extended Theorem \ref{thm-PLS} to functions
	which, instead of being band-limited (i.e. $\mathcal F f$ is compactly supported),  have sufficiently fast decaying Fourier transforms.  The rate of decay was characterised in terms of a weight function.  Precisely,  they proved the following theorem.
	
	\begin{thm} \cite[Theorem 1.3]{JM} \label{thm-JM}
		Suppose a weight $W :[0, \infty) \ra [0, \infty]$ satisfies
		\begin{enumerate}
			\item[(i)] $W(0)=1$, $W$ is non-decreasing, $W$ is lower semi-continuous, and $\lim_{r\ra \infty} W(r) = \infty$.
			
			\item[(ii)] The mapping $\log r \ra \log W(r)$ is convex on $[1, \infty)$.
			
			\item[(iii)] The following integral
			\be \label{decay-weight}
			\int_{0}^\infty\frac{\log W(r)}{1+r^2}~dr=\infty.
			\ee
		\end{enumerate}
		Then,  for every $d \in \N, \gamma \in (0, 1), l > 0$ and $C_ W > 0$, there exists a finite constant $C = C(d, W, C_W , \gamma, l) > 0$ such that if $f \in L^2(\R^d)$ satisfies
		\be \label{est-hatf}
		\int_{\R^d} |\mathcal{F} f(\xi)|^2 ~W(|\xi|)^2~d \xi \leq C_W^2 \|f\|_{L^2(\R^d)}^2,
		\ee
		and $E$ is a $(\gamma, l)$-relatively dense set,  then
		\bes
		\|f\|_{L^2(\R^d)}\leq C \|f\|_{L^2(E)}.
		\ees
	\end{thm} 
	
	\begin{rem}
		\begin{enumerate}
			\item[(i)]  The result is sharp in the sense that,  if the integral (\ref{decay-weight}) is finite,  then the Paley-Wiener theorem implies that there exists compactly supported function supported in an arbitrarily small ball; therefore,  the conclusion of the theorem  fails to hold. 
			
			\item[(ii)] 
			Theorem \ref{thm-JM} differs from Theorem \ref{thm-PLS} in that the functions under consideration are no longer analytic, and thus complex-analytic techniques cannot be employed. Its formulation was inspired by the breakthrough uniqueness theorem of Bourgain and Dyatlov \cite{BD}, where the result was established in one dimension for the weight $W(t) = e^{t/\log^\delta(e+t)}$, with $0<\delta<1$. The key observation underlying their proof is that if $\mathcal F f(\xi)$ decays at infinity significantly faster than $\exp(-\xi/\log|\xi|)$, then the function is quasi-analytic. In the one-dimensional setting, the main tool is the use of harmonic measure on a strip \cite{BD}, while in higher dimensions,  Jaye and Mitkovski \cite{JM} instead employed a Remez-type inequality for quasi-analytic functions, established in \cite{NSV}.
			
			\item[(iii)] For dimension one,  in \cite[Proposition 1.4]{JM},  the authors have explicitly computed the constant: if $E$ is $(\gamma, 1)$-relatively dense set, then for $W(t)=e^{t/\log(e+t)}$,  the constant can be choosen $C=(A/\gamma)^{(\log A C_W)^{e^A}}$, for some absolute constant $A>0$.  
			
			\item[(iv)] A similar quantitative uniqueness result for functions in a general class of Gelfand–Shilov spaces on $\R^d$ was obtained in \cite{M}.
			
		\end{enumerate}
	\end{rem}
	
	From the above remark, it is evident that Theorem \ref{thm-JM} concerns quasi-analytic functions. Since we have already defined a class of quasi-analytic functions on compact Lie groups $G$, it is natural to seek an analogous uniqueness result in this setting. The $Ad$-invariant inner product on $\mathfrak g$ induces a bi-invariant metric, making $G$ a Riemannian manifold. For $g \in G$ and $r > 0$, let $\mathcal B(g,r)$ denote the geodesic ball of radius $r$ centered at $g$. We write $r_{\mathrm{inj}} > 0$ for the injectivity radius of $G$ and $\mathfrak R\in (0, r_{inj})$ for which the integral formula (\ref{vol-est}) is valid.  
	
	\begin{defn} \label{defn-dense} Let $\gamma \in (0, 1]$.  A subset $\Omega \subset G$ is said to be $\gamma$-relatively dense if there exists $r \in (0,  \mathfrak R/2)$ such that  
		\bes
		\mu\left(\mathcal B(g,  r) \cap \Omega \right) \geq \gamma ~\mu\left(\mathcal B(g,  r) \right),  \:\: \text{  for all } \:\: g\in G. 
		\ees 
	\end{defn}
	As our goal is to establish a uniqueness result for a class of quasi-analytic functions with a uniform constant, we refine the definition of the quasi-analytic class. Specifically,  given a quasi-analytic sequence 
	$\mathcal M=\{M_n\}_{n=0}^\infty$ and $p\in [1, \infty]$,  we define $C_{\mathcal M,  p}(G)$ by 
	\bes
	C_{\mathcal M, p}(G)=\left\{f\in C^\infty(G): \|\mathcal L^nf\|_{L^p(G)} \leq M_{2n},  \:\: \forall n \in \N_0 \right\}.
	\ees
	The following is our main uniqueness result on compact,  connected Lie groups $G$.
	\begin{thm} \label{thm-main}
		Let $\gamma\in(0, 1), p\in [1,  \infty]$ and $\mathcal M=\{M_n\}_{n=0}^\infty$ be a non-decreasing quasi-analytic sequence.  There exists an absolute constant $C=C(G)$ such that 
		\bes
		\|f\|_{L^p(G)}\leq  \left( \frac{C \Gamma \left(2n_{\widetilde{\mathcal M},  s} \right)}{\gamma} \right)^{2 n_{\widetilde{\mathcal M},  s}+\frac{1}{p}}   \|f\|_{L^p(\Omega)},  \:\: \textit{ for all } f\in   C_{\mathcal M,  p}(G),
		\ees
		for all  $\gamma$-relatively dense subset $\Omega$ in $G$.  Here,  $\widetilde{\mathcal M}$ and $s$ are defined in (\ref{defn-tilde-M}) and (\ref{defn-s}) respectively.  Also,  the functions $n_{\mathcal M,  s},$ and  $\Gamma$  are defined in (\ref{defn-Gamma}) and (\ref{defn-ns}) respectively.
	\end{thm}
	
	\begin{rem}
		The above estimates depend solely on 
		$\gamma$,  and are independent of the position or shape of the relatively dense set $\Omega$.
	\end{rem}
	As a consequence of Theorem \ref{thm-main}, we derive the following analogue of Theorem \ref{thm-JM} in the setting of $G$.  For the notation used, we refer the reader to Section 2.
	\begin{cor} \label{thm-main-2}
		Suppose a weight $W :[0, \infty) \ra [0, \infty]$ satisfies the hypothesis $(i)$,  $(ii)$ and $(iii)$ of Theorem \ref{thm-JM}.  Then,  for every $\gamma\in (0, 1]$ and $C_W>0$, there exists a finite constant $C=C(\gamma, C_W,  G)>0$ such that, whenever a function $f\in L^2(G)$ satisfies 
		\be \label{weight-G}
		\sum_{[\xi] \in \widehat G}d_\xi ~W(|\xi|)^2~ \|\widehat f(\xi)\|_{HS}^2 \leq C_W^2 \|f\|_{L^2(G)}^2,
		\ee
		and $\Omega$ is a $\gamma$-relatively dense set,  we have 
		\bes\|f\|_{L^2(G)}\leq C\|f\|_{L^2(\Omega)}.
		\ees
	\end{cor}
	Theorem \ref{thm-main} further implies the following spectral inequality on $G$.  For $\Lambda>0$,  let $P_\Lambda= \chi_{[0,  \Lambda]}(\mathcal L)$ denote the spectral projection up to energy $\Lambda>0$ for the operator $\mathcal L$; see (\ref{defn-PLambda}) for the definition.
	\begin{cor} \label{thm-spectral}
		Let $\gamma \in (0, 1]$ and $\Lambda>0$.  There exists $C=C(G)>0$ such that for every $\gamma$-relatively dense set $\Omega$
		\bes
		\|P_\Lambda f\|_{L^2(G)}\leq  \left( \frac{C}{\gamma} \right)^{2e \sqrt{\Lambda}+1/2}   \|P_\Lambda f\|_{L^2(\Omega)},  \:\: \textit{ for all } f\in   L^2(G).
		\ees 
	\end{cor}
	As an application of the preceding spectral inequality,  we derive the following observability theorem on $G$.  For the proof,  we refer the reader,  for example,  to \cite[Lemma 2.3]{R}.
	\begin{cor}
		Let $\Omega \subset G$ be a $\gamma$-relatively dense set.  For any solution $u$ of the heat equation 
		\bes
		\partial_{t}u(t, g) -\mathcal{L} u(t, g)=0 \; in \; (0,+\infty)\times G, \; u(0,.) = f \in L^2(G),
		\ees
		and for any $T>0$, there exists a positive constant $C_T$ such that
		\bes
		\int_{G} |u(T,g)|^2~d\mu(g) \leq C_{T} \int_{0}^{T} \int_{\Omega}|u(t,g)|^2 ~d\mu(g) dt.
		\ees
	\end{cor}

	We now consider the homogeneous spaces $G/H$ of $G$, where $H$ is any closed subgroup of $G$.   For $x\in G/H$ and $r>0$, let $\widetilde{\mathcal B}(x, r)$ denote the geodesic ball of radius $r$ centred at $x\in G/H$.  Let $\mathfrak r$ be a positive number for which the  volume form has the estimate (\ref{vol-est-GH}) for all $0< r< \mathfrak r$ (see Section 4).  Following Definition \ref{defn-dense},  we define a relatively dense set in $G/H$ as follows: a subset $\widetilde{\Omega} \subset G/H$ is said to be $\gamma$-relatively dense if there exists $r\in (0, \mathfrak r/2)$ such that  
	\bes
	\widetilde{\mu}\left(\widetilde{\mathcal B}(x,  r) \cap \widetilde{\Omega} \right) \geq \gamma ~\widetilde{\mu}\left(\widetilde{\mathcal B}(x,  r) \right),  \:\: \text{  for all } \:\: x\in G/H. 
	\ees 
	We define the class of quasi-analytic functions on $G/H$ associated with a quasi-analytic sequence $\mathcal M=\{M_n\}_{n=0}^\infty$ by
	\bes
	C_{\mathcal M, p}(G/H)=\left\{f\in C^\infty(G/H): \| \tilde{\mathcal L}^n f\|_{L^p(G/H)} \leq M_{2n},  \:\: \forall n \in \N_0 \right\}.
	\ees
	We now state our main uniqueness result on $G/H$.
	\begin{thm} \label{thm-X}
		Let $\gamma\in(0, 1] ,  p \in [1, \infty]$ and $\mathcal M=\{M_n\}_{n=0}^\infty$ be a non-decreasing quasi-analytic sequence.  There exists $C=C(G/H)$ such that 
		\bes
		\|f\|_{L^p(G/H)}\leq 
		\left( \frac{C \Gamma \left(2n_{\widetilde{\mathcal M},  s} \right)}{\gamma} \right)^{2 n_{\widetilde{\mathcal M}, s}+\frac{1}{p}} ~\|f\|_{L^p(\widetilde{\Omega})},  \:\: \textit{ for all } f\in   C_{\mathcal M, p}(G/H),
		\ees
		for all  $\gamma$-relatively dense subset $\widetilde{\Omega}$ in $G/H$.
	\end{thm}
	
	For $\Lambda>0$,  let $\widetilde{\mathcal P}_\Lambda= \chi_{[0,  \Lambda]}(\widetilde{\mathcal L})$ denote the spectral projection up to energy $\Lambda>0$ for the operator $\widetilde{\mathcal L}$. 
	As a consequence of Theorem \ref{thm-X}, we obtain the following spectral estimates, which extend the result of \cite{DV} on the sphere to the more general setting of homogeneous spaces.
	\begin{cor}
		Let $\gamma \in (0, 1]$ and $\Lambda\geq 0$. There exists $C=C(G/H)$ such that for all $\gamma$-relatively dense sets $\widetilde{\Omega}\subset G/H$,  we have
		\bes
		\|\widetilde{\mathcal P}_\Lambda f\|_{L^2(G/H)}\leq \left(\frac{C}{\gamma}\right)^{2e \sqrt{\Lambda}+1/2} \|\widetilde{\mathcal P}_\Lambda f\|_{L^2(\widetilde{\Omega})}, \:\: \textit{ for all } \:\: f\in L^2(G/H).
		\ees
	\end{cor}
	
	\section{Preliminaries}
	
	In this section,  we provide a brief review of the necessary preliminaries and Fourier analysis on compact Lie groups.  We adopt the standard notation and refer to \cite{H1,  H2,  RT-B} for further details.
	
	Let $G$ be a compact, connected Lie group of dimension $d$ and let $\mathfrak g$ denote its Lie algebra,  identified with the tangent space $T_e(G)$ at the identity element $e$ in $G$.  Let $Ad$ denote the adjoint action of $G$ on $\mathfrak g$.  Since $G$ is compact,  there exists an $Ad$-invariant inner product $\langle \cdot, \cdot \rangle$ on $\mathfrak g$.  Let $\{X_1, \cdots,  X_d\}$ be an orthonormal basis of $\mathfrak g$ with respect to this inner product.   The Casimir operator $\mathcal L$ is then defined by 
	\bes
	\mathcal L= \sum_{i=1}^d X_i^2,
	\ees
	where each $X_i$ is viewed as a left-invariant differential operator on $G$.  The operator $\mathcal L$ depends on the choice of the invariant inner product $\langle \cdot,  \cdot \rangle$,  but not on the choice of the orthonormal basis.   This operator coincides with the Laplace-Beltrami operator associated with the bi-invariant Riemannian metric on $G$ induced from $\langle \cdot,  \cdot \rangle$.  If $G$ is assumed to be semisimple, one typically takes the $Ad$-invariant inner product, and hence the corresponding bi-invariant metric,  to be derived from the Killing form of $\mathfrak g$.  Let $\mu$ denote the normalized Haar measure on $G$,  that is,  the unique bi-invariant,  Borel regular probability measure on $G$.

	We denote by $\exp: \mathfrak g \ra G$ the Lie group exponential map.  This coincides with the Riemannian exponential $Exp_e$ at identity for the induced Riemannian connection on $G$.  For $X\in \mathfrak g$,  let $\widetilde X$  be the corresponding left-invariant vector field, so that $\widetilde X(g) \in T_gG$ for each $g\in G$.   Then,  for any $g$ in $G$,  one has $Exp_g(\widetilde X(g))= g \exp(X)$.  For $X\in \mathfrak g$,  we write $\|X\|$ for the norm $\mathfrak \langle X,  X \rangle^{1/2}$.  Let $B(0, r)=\{X\in \mathfrak g: \|X\|< r\}$ be the open ball centred at $0$ in $\mathfrak g$.  Since $G$ is compact,  the injectivity radius  $r_{inj}$ is strictly positive.  Then for each $0<r< r_{inj}$,  the exponential map $Exp_g: B(0, r) \ra \mathcal B(g, r)$ is a surjective and analytic diffeomorphism. 
	
	The volume form $\mu$ in geodesic polar coordinates around $g$ is given by 
	\bes
	d\mu(Exp_g( rX))= d\mu(g \exp(X))= J(r, X)~dr~d\sigma(X), 
	\ees
	where $d\sigma(X)$ denotes the canonical normalized measure of the unit sphere $\mathcal S^{d-1}$ in $\mathfrak g$.  The Jacobian $J(r, X)$ is given by
	\bes
	J(r, X)= r^{d-1}~\left(1- \frac{1}{6}R(X)r +o(r^2)\right),
	\ees
	where  $R(X)$ is Ricci curvature evaluated on the vector $X$ \cite{GHL}.  Since $R(X)$ is bounded on $\mathcal S^{d-1}$,  there exists $\mathfrak R \in (0, r_{inj})$ such that if $r<\mathfrak R$,  we have $J(r, X) \sim r^{d-1}~dr~d\sigma(X)$.  Consequently,   there exist positive constants $\mathcal C,  \mathcal C^\prime$ (independent of $g$) such that for  $0< r< \mathfrak R$ and for any measurable function $f$
\bea \label{vol-est}
&& \mathcal C^\prime \int_{\mathcal S^{d-1}}~\int_{0}^{r}~f\left(g \exp(tX) \right)~t^{d-1}~dt ~ d\sigma(X) \leq	\int_{\mathcal B(g,  r)} f(g)~d\mu(g) \nonumber\\
&& \hspace{3cm} \leq \mathcal C \int_{\mathcal S^{d-1}}~\int_{0}^{r}~f\left(g \exp(tX) \right)~t^{d-1}~dt ~ d\sigma(X).
	\eea  
In particular,  for all $g\in G$ and $0<r< \mathfrak R$
	\be  \label{est-volume}
	\mathcal C^\prime r^d \leq \mu(\mathcal B(g,  r)) \leq \mathcal C r^d.
	\ee
	
	Every element of $\mathfrak g$, and hence of its universal enveloping algebra $\mathcal U(\mathfrak g)$ gives rise to both left and right-invariant differential operators on $G$.  More precisely,  for a function $f\in C^\infty(G)$, an element $D=X_1 X_2\cdots X_j \in \mathcal U(\mathfrak g)$ with $X_1, \cdots, X_j \in \mathfrak g$ and $g\in G$,  we define 
	\bea \label{defn-dr}
	&& \widetilde{D}f(g)=f(g; D)= \left(\partial^j/ \partial t_1 \partial t_2 \cdots \partial t_j\right)_0f(g \exp t_1X_1 \cdots \exp t_jX_j) \nonumber,\\
	&& d_r(D)f(g)=f(D; g)= \left(\partial^j/ \partial t_1 \partial t_2 \cdots \partial t_j\right)_0 f(\exp t_1X_1 \cdots \exp t_jX_j g),
	\eea
	(the suffix $0$ in the right-hand side denotes that the derivatives are taken at $t_1= \cdots =t_j=0$).  Thus, one can identify $\mathcal U(\mathfrak g)$ with $\mathbb D(G)$, the algebra of left-invariant differential operators on $G$. Moreover, it is known that $\mathcal L$ lies in the center of $\mathcal U(\mathfrak g)$.

	Let $\widehat G$ denote the unitary dual of $G$,  that is,  the set of equivalence classes of strongly continuous irreducible unitary representations of $G$. 
	Let $[\xi] \in \widehat G$ denote the equivalence class of an irreducible unitary representation $\xi :G\ra \mathcal U(H_\xi)$.  Here,  $H_\xi$  denotes the representation space of $\xi$, which is finite-dimensional,  and we denote the dimension of $H_\xi$ by $d_\xi$.  For a representation $\xi$,  by choosing an orthonormal basis in the representation space $H_\xi$,  we can view $\xi$ as a matrix-valued function $\xi : G \ra M_{d_\xi}(\C)$.  Let $\xi_{ij}(g)$,  $1\leq i, j\leq d_\xi$,  be the matrix coefficients of the representation $\xi$.   
	
	Let $[\xi]\in \widehat G$. Then $\xi$ is a differentiable function on $G$ with values in $\mathcal U(H_\xi)$ \cite[Ch.  II, Thm.  2.6]{H1}.  Moreover,  for $X\in \mathfrak g$
	\be \label{rel-X}
	\widetilde X \xi(g)= \frac{d}{d t} \Big |_{t=0} \xi \left(g \exp(tX)\right)= \xi(g) ~\frac{d}{d t} \Big|_{t=0} \xi\left( \exp (tX)\right)= \xi(g)~d\xi(X),
	\ee
	where $d\xi$ denotes the derived representation of $\xi$.  We extend the representation $d\xi$ of $\mathfrak g$ to its universal algebra $\mathcal U(\mathfrak g)$.  More generally,  for any 
$D \in \mathcal U(\mathfrak g)$,  we then have  
	\be \label{rel-D}
	\widetilde{D} \xi(g)= \xi(g) ~ d\xi(D),  \:\: \textit{ for } g\in G.
	\ee
	Similarly,  for the right-invariant differential operator $d_r(D)$
	\be \label{rel-D-r}
	d_r(D) \xi(g)= d\xi(D)~\xi(g).
	\ee
	Since $\mathcal L$ is in the centre of $\mathcal U(\mathfrak g)$,  by Schur's lemma $d\xi(\mathcal L)$ acts as a constant for any irreducible representation $\xi$.   Therefore, 
	\be \label{rel-L}
	d\xi(\mathcal L) =\sum_{i=1}^d d\xi(X_i^2)=\sum_{i=1}^d d\xi(X_i)^2= -\lambda_{\xi} I,
	\ee
	or some $\lambda_\xi\geq 0$.  Since $\xi$ is unitary, $d\xi(X_i)$ is skew-Hermitian and hence the constant should be negative.  It now follows from the relations (\ref{rel-D}) and (\ref{rel-L}) that for each $[\xi] \in \widehat G$,  the matrix elements $\xi_{ij}(g)$ of $\xi$ are the eigenfunctions for the Laplace-Beltrami operator $\mathcal L$ with the same eigenvalue $-\lambda_{\xi}$,  that is, 
	\bes
	-\mathcal L \xi_{ij}(g) = \lambda_{\xi} \xi_{ij}(g), \:\: \textit{ for all } \:\: 1\leq i, j \leq d_\xi.
	\ees
	
	We denote by $\widehat{G}^\ast$ the class of equivalence representations from $\widehat G$ excluding the trivial one.  For $[\xi] \in \widehat G$,  we write $|\xi| = \sqrt{ \lambda_{\xi}} \geq 0$.  Thus,  $|\xi|>0$ for $[\xi] \in \widehat G^\ast$,  while $|\xi| = 0$  when $[\xi]$ corresponds to the trivial representation. 
We further set $\langle \xi \rangle = \left(1 + \lambda_{\xi}\right)^{1/2}$.  The following lemma shows that the growth of the dimension $d_\xi$ is controlled by the eigenvalue $|\xi|^2$.
	\begin{lem} \cite[Lemma 3.1]{DR} \label{lem-d-xi} The series $\sum_{[\xi] \in \widehat G} ~d_\xi^2 ~ \langle \xi \rangle^{-2t} < \infty$ if and only if $t > d/
		2$.
	\end{lem}
	
The group Fourier transform of $f\in  L^1(G)$
at a representation $\xi$, with $[\xi] \in \widehat G$, is the operator $\widehat f(\xi) \in End(H_\xi )$ 
defined by	
\bes
	\widehat f(\xi) = \int_{G} f(g)~\xi(g)^\ast~d\mu(g),
	\ees
	where $\xi(g)^\ast$ denotes  the adjoint of the unitary matrix $\xi(g)$.  For $f\in L^2(G)$,  the Peter-Weyl theorem implies the following Fourier inversion formula \cite[Cor. 10.2.10]{RT-B}
	\be \label{FT-inversion}
	f(g)= \sum_{[\xi]\in \widehat G} d_{\xi} Tr\left(\xi(g)~\widehat f(\xi) \right),
	\ee
	for almost every $g\in G$ as well as in $L^2(G)$.  If $f\in C^\infty(G)$,  the Fourier series (\ref{FT-inversion}) converges uniformly on $G$ \cite[Ch. V,  \$ 2-3]{H2}.  The Plancherel formula takes the form
	\be \label{Plancherel}
	\|f\|_{L^2(G)} ^2= \sum_{[\xi]\in \widehat G} d_\xi~ Tr \left(\widehat f(\xi)  \widehat f(\xi)^\ast \right),
	\ee
	\cite[Prop. 10.3.17]{RT-B}.   For $\Lambda>0$,  we define the spectral projection  operator $P_{\Lambda}= \chi_{[0, \Lambda]}(\mathcal L)$ up to energy $\Lambda$ by
	\be \label{defn-PLambda}
	P_\Lambda f (g)= \sum_{\{[\xi]\in \widehat G: |\xi|\leq \sqrt{\Lambda}\}} d_\xi Tr\left(\xi(g) \widehat f(\xi)\right),  
	\ee
	for all $f\in L^2(G)$.  
	
	\section{Uniqueness results on compact Lie groups}
	
	In this section,  we prove Theorem \ref{thm-main}. Throughout this section,  $\mathcal N  = \{N_n\}_{n=0}^\infty$ be any log-convex sequence satisfying the quasi-analyticity condition (\ref{quasi-cond}) and $N_0=1$.  For the interval $I=[0, 1]$, we refine the definition of the quasi-analytic class.  Precisely,
	\bes
	C_{\mathcal N} [0, 1]=\{h\in C^\infty[0, 1]: \|h^{(n)}\|_{L^\infty[0, 1]}\leq N_n, \:\: \forall n\in \N_0\}.
	\ees 
	For $h \in C_{\mathcal N} [0, 1]$, the Bang degree $n_h$ is defined by
	\bes
	n_h= \max \left\{k\in \N : \sum_{\log \|h\|_{L^\infty[0,1]}^{-1}<n \leq k} \frac{N_{n-1}}{N_n} < e \right\}.
	\ees
	It is well-known that the Bang degree controls the number of zeros of a function in $C_{\mathcal N}[0, 1]$,  counting multiplicities, therefore,  plays the same role as the degree of polynomials \cite{NSV}.  Clearly, $n_h$ depends on both the decay of the ratios of $N_{n-1}/N_n$ and a lower bound for $\|h\|_{L^\infty[0, 1]}$. Since $\mathcal N$ satisfies (\ref{quasi-cond}), the quantity $n_h$ is always finite.  For our purposes, we will want uniform bounds on $n_h$ for arbitrary functions $h \in C_\mathcal N[0,  1]$. To achieve this,  we set, for $s \in (0, 1]$
	\be \label{defn-Gamma}
	n_{\mathcal N, s}= \max \left\{k\in \N: \sum_{-\log s <n \leq k} \frac{N_{n-1}}{N_n} < e \right\}.
	\ee
	Therefore, if $h \in C_\mathcal N[0, 1]$ satisfies $\sup_{t \in [0, 1]} |h(t)| \geq s$, then $n_h\leq n_{\mathcal N, s}$.  As in  \cite{JM, NSV}, we also define
	\be \label{defn-ns}
	\gamma_{\mathcal N}(n)= \sup_{1\leq l\leq n} l\left(\frac{N_{l+1}N_{l-1}}{N_l^2}-1 \right), \:\: \textit{ and } \Gamma(n)=4e^{4+4\gamma_{\mathcal N}(n)}.
	\ee
	We recall the following result of Nazarov-Sodin-Volberg \cite[Theorem B]{NSV}.
	\begin{thm} \label{thm-NSV}
		Suppose that $h\in C_{\mathcal N}[0, 1]$. Then for any interval $I \subset [0, 1]$ and measurable set $E \subset I$ with $| E| > 0$, we have
		\bes
		\|h\|_{L^\infty(I)}\leq \left(\frac{\Gamma(2n_h) |I|}{|E|}\right)^{2n_h} \|h\|_{L^\infty(E)}.
		\ees
	\end{thm}
	To prove Theorem \ref{thm-main},  we require the following result.  
	\begin{lem} \label{lem-Sobolev}
		There exists $\mathfrak C>1$ such that for all $f\in C^\infty(G)$,  $X\in \mathcal S^{d-1}$ and for all $n\in \N_0$,  we have
		\beas
		\| \widetilde X^n f \|_{L^\infty(G)} &\leq &  \mathfrak C^{n+1} 
		\begin{cases} 
			\|\mathcal L^{d+n/2} f\|_{L^1(G)},  & \text{ if $n$ is even }; \\
			\|\mathcal L^{d+(n-1)/2} f\|_{L^1(G)}^{1/2}~\|\mathcal L^{d+(n+1)/2} f\|_{L^1(G)}^{1/2},  & \text{ if $n$ is odd }.
		\end{cases}
		\eeas
	\end{lem}
	
	\begin{proof}
		Let $f\in C^\infty(G)$ and $X\in \mathcal S^{d-1}$.  It follows from the inversion formula (\ref{FT-inversion}) and the relation (\ref{rel-D}) that for $g\in G$ and $n\in \N_0$
		\bea \label{est-X}
		|\widetilde X^n f (g)| &=& |\sum_{[\xi] \in \widehat{G}} ~d_{\xi} ~Tr \left((\widetilde X^n \xi)(g) ~ \widehat{f}(\xi) \right)| \nonumber \\
		&\leq&  \sum_{[\xi] \in \widehat{G}} ~d_{\xi} ~\|\xi(g)~d\xi(X^n) \|_{HS} ~\| \widehat{f}(\xi)\|_{HS} \nonumber \\
		&\leq&  \sum_{[\xi] \in \widehat{G}} ~d_{\xi} ~\|d\xi(X^n) \|_{op}~\|\xi(g)\|_{HS} ~\| \widehat{f}(\xi)\|_{HS}.
		\eea
		First,  we compute the operator norm of $d\xi(X^n)$.  
		We recall that $\{X_1, \cdots, X_d\}$ is an orthonormal basis of $\mathfrak g$.  Since $d\xi(X_i)$ is skew-Hermitian for each $i=1, \cdots,  d$,  it follows from the relation (\ref{rel-L}) that  for any $v\in H_\xi$,
		\bes
		\sum_{i=1}^d \| d\xi(X_i)v\|^2 =- \sum_{i=1}^d \langle d\xi(X_i)^2v, v \rangle =  -\langle d\xi(\mathcal L) v, v \rangle = \lambda_\xi \|v\|^2.
		\ees
		Hence,  $\|d\xi(X_i)\|_{op}\leq |\xi|$,  for each $i$.  Since $X\in \mathcal S^{d-1}$,  there exist scalars $c_i$ with $\sum_{i=1}^d c_i^2=1$ such that $X=\sum_{i=1}^dc_iX_i$.  It follows that 
		$\|d\xi(X)\|_{op} \leq \sqrt d ~ |\xi |$ and consequently, 
		\be \label{est-X-1}
		\|d\xi(X^n)\|_{op}\leq \|d\xi(X)\|_{op}^n\leq d^{n/2} |\xi|^n, \:\: \textit{ for all } \:\:  n\in \N_0.
		\ee
		Since $\xi(g)$ is unitary,  it follows that $\|\xi(g)\|_{HS}=\sqrt d_{\xi}$.  Consequently,  in view of (\ref{est-X-1}),  the estimate (\ref{est-X}) yields
		\be \label{est-X-2nd}
		|\widetilde X^n f(g)| \leq d^{n/2} \sum_{[\xi] \in \widehat G}~d_\xi^{3/2}~|\xi|^n~ \|\widehat{f} (\xi) \|_{HS}.
		\ee
		It follows from the definition of the Fourier transform that 
		\bes
		\|\widehat f(\xi)\|_{HS} \leq \int_G |f(g)|~\|\xi(g)^\ast\|_{HS}~d\mu(g) \leq \sqrt d_\xi \|f\|_{L^1(G)}.
		\ees  
		Therefore,  the relation $\widehat{\mathcal L^{n} f}(\xi) = |\xi|^{2n} \widehat f(\xi)$ yields
		\be \label{est-HS}
		|\xi|^{2n}~\|\widehat f(\xi)\|_{HS} = \|\widehat{\mathcal L^n f}(\xi)\|_{HS} \leq \sqrt d_\xi \|\mathcal L^n f\|_{L^1(G)}.
		\ee
		Let $\lambda_1 > 0$ be the smallest positive eigenvalue of  $- \mathcal{L}$.  It is easy to show that, for $[\xi] \in \widehat{G}^\ast$
		\bes
		|\xi|^{-k} \leq {C_1}^{k} {\langle \xi \rangle}^{-k}, \; \textit{ where } \: \: C_1 = \left(1 + \frac{1}{\lambda_1} \right)^{1/2}, \:\:  k\in \N_0.
		\ees
		We first consider the case $n$ is even.  Using the estimate (\ref{est-HS}),  it follows from (\ref{est-X-2nd}) that
		\beas
		|\widetilde X^n f(g)| &\leq& d^{n/2}C_1^{2d}~\sup_{\xi \in \widehat G}  \left( d_\xi^{-1/2}~|\xi |^{2d+n}~  \|\ \widehat{f} (\xi) \|_{HS}\right)~ \sum_{[\xi] \in \widehat{G}^\ast}~d_\xi^2~ \langle \xi \rangle^{-2d}\\
		&\leq& d^{n/2}C_1^{2d}~C_d~ \|\mathcal L^{d+n/2} f\|_{L^1(G)},
		\eeas
		where $C_d=\sum_{[\xi] \in \widehat{G}^\ast}~d_\xi^2~ \langle \xi \rangle^{-2d}$,  which is finite by Lemma \ref{lem-d-xi}. 
		If $n$ is an odd integer,  applying H\"older's inequality on the right-hand side of (\ref{est-X-2nd}),  we get that
		\beas
		|\widetilde X^n f(g)| &\leq& d^{n/2} \left( \sum_{[\xi] \in \widehat{G}^\ast}~d_\xi^{3/2}~|\xi|^{n-1}~ \|\widehat{f} (\xi) \|_{HS} \right)^{1/2}~ \left(\sum_{[\xi] \in \widehat{G}^\ast}~d_\xi^{3/2}~|\xi|^{n+1}~ \|\widehat{f} (\xi) \|_{HS}\right)^{1/2}\\
		& \leq&  d^{n/2}~C_1^{2d}~C_d~ \|\mathcal L^{d+(n-1)/2} f\|_{L^1(G)}^{1/2}~\|\mathcal L^{d+(n+1)/2} f\|_{L^1(G)}^{1/2}.
		\eeas
		This completes the proof.
	\end{proof}
	
	\begin{proof}[Proof of Theorem \ref{thm-main}]
		Without loss of generality,  we assume $\|f\|_{L^p(G)}=1$.  Since $\Omega$ is $\gamma$-relatively dense set,  there exists $r \in (0,  \mathfrak R/2)$ such that $\mu(\mathcal B(g, r) \cap \Omega) \geq \gamma \mu(\mathcal B(g, r))$. Without loss of generality, we may assume that $r \leq 1$.  Clearly,  for every $g\in G$,  the exponential map  $Exp_g: B(0, r) \ra \mathcal B(g,  r)$ is surjective.  We find a finite collection $g_j$ in $G$ such that $G=\cup_{j=1}^N \mathcal B_j(g_j,  r)$. From now onwards we use $\mathcal B_j$ to denote $\mathcal B(g_j, r)$.  We note that 
		\be \label{est-GB}
		N^{-1} \sum_{j=1}^N \|f\|_{L^p(\mathcal B_j)} \leq  \|f\|_{L^p(G)} \leq \sum_{j=1}^N  \|f\|_{L^p(\mathcal B_j)}.
		\ee
		We first consider the case $p\in [1, \infty)$.  Let $B$ be a sufficiently large number and its exact value will be chosen later.  A ball $\mathcal B_j$ is said to  be bad if there exists $n\in \N$ and $X_j\in \mathcal S^{d-1}$ such that 
		\bes
		\int_{\mathcal B_j} \left| \widetilde X_j^{n} f(g)\right|^p ~d\mu(g) > B^{n} \bar M_{n}^p \int_{\mathcal B_j} |f(g)|^p~d\mu(g),
		\ees
		where
		\be \label{defn-barM}
	    \Bar M_n  =
		\begin{cases} 
		 M_{2d+n},  & \text{ if $n$ is even }; \\
		 \sqrt{M_{2d+n-1} M_{2d+n+1}},  & \text{ if $n$ is odd }.
		\end{cases}
		\ee
		If a ball is not bad,  we call it good.  Precisely,  if $\mathcal B_j$ is good,  then for all $n\in \N$ and $X \in \mathcal S^{d-1}$ 
		\be \label{defn-good}
		\int_{\mathcal B_j} \left| \widetilde X^{n} f(g)\right|^p ~d\mu(g) \leq B^{n} \bar M_{n}^p \int_{\mathcal B_j} |f(g)|^p~d\mu(g).
		\ee
		It follows from the definition that if $\mathcal B_j$ is bad,  then 
		\be \label{est-bad}
		\int_{\mathcal B_j}  |f(g)|^p~d\mu(g) < \sum_{n\in \N}\frac{1}{B^{n}~\bar M_{n}^p}  \int_{\mathcal B_j}  |\widetilde X_j^{n} f(g)|^p~d\mu(g).
		\ee
		We first observe that since $f\in C_{\mathcal M, p}(G)$, and $\mu$ is a normalised measure on $G$, we get that $\|\mathcal L^n f\|_{L^1(G)}\leq M_{2n}$. Consequently,  using Lemma \ref{lem-Sobolev}, we obtain that for all $X\in \mathcal S^{d-1}$
		\be \label{est-bad-union}
		\int_{\mathcal B_j}  |\widetilde X^{n} f(g)|^p~d\mu(g)\leq \mu(\mathcal B_j)~ \|\widetilde X^{n} f\|_{L^\infty(\mathcal B_j)}^p\leq  \mathfrak C^{p(n+1)} ~ \bar M_n^p.
		\ee
		Let $\mathcal B_{bad}=\cup \mathcal B_j$,  where the balls $\mathcal B_j$'s are bad.  Similarly,  $\mathcal B_{good}= \cup \mathcal B_j$, where $\mathcal B_j$ are good balls. Then, the estimates (\ref{est-bad-union}) yields
		\beas
		\int_{\mathcal B_{bad}}  |f(g)|^p~d\mu(g) &\leq&  \sum_{\{j: \mathcal B_j \textit{ is bad }\}} \int_{\mathcal B_j} | f(g)|^p~d\mu(g)\\
		&\leq&  \sum_{n\in \N}\frac{1}{B^{n}~\bar{M}_{n}^p} \sum_{\{j: \mathcal B_j \textit{ is bad }\}} \int_{\mathcal B_j} | \widetilde X_j^n f(g)|^p\\
		&\leq& N \sum_{n\in \N} \frac{\mathfrak C^{p(n+1)} }{B^{n}}\leq \frac{1}{2} \|f\|_{L^p(G)}^p,
		\eeas
		by choosing $B$ large enough. Consequently,
		\be \label{est-good}
		\int_{\mathcal B_{good}} |f(g)|^p~ d\mu(g)  \geq \frac{1}{2} \|f\|_{L^p(G)}^p.
		\ee
		We now fix a good ball $\mathcal B_j$,  and henceforth we will denote it by $\mathcal B$,  omitting the subscript $j$.  Let $g_0 \in \mathcal B$ be a point such that $|f(g_0)| \geq \|f\|_{L^p(\mathcal B)} ~\mu(\mathcal B)^{-1/p}$.  This is possible and can be shown by contradiction.  For $X\in \mathcal S^{d-1}$,  we define the geodesic curve segment $c_X(t)= g_0 \exp (tX)\subset \mathcal B$ starting at $g_0$.  Precisely,  we have $c_X(0)=g_0$ and $\dot{c}_X(0)=dL_g X_e=\widetilde X(g_0)$.  Applying estimate (\ref{vol-est}),  we get that 
		\beas
		\mu(\mathcal B\cap \Omega ) \leq \mathcal C~\int_{\mathcal S^{d-1}}~\int_{0}^{2r} \chi_{\mathcal B\cap \Omega}(c_X(t))~t^{d-1}~ dt~ d\sigma(X).
		\eeas
		Consequently,  there exists $Y \in \mathcal S^{d-1}$ such that
		\be \label{est-B-00}
		\mu(\mathcal B\cap \Omega ) \leq \mathcal C \int_{0}^{2r} \chi_{\mathcal B\cap \Omega}(c_{Y}(t))~t^{d-1}~ dt.
		\ee 
Let $l=\sup\{t\in[0,2r]: c_Y(t)\in \mathcal B\}$ and $I=\{c_Y(t): 0\leq t < l\} \subset \mathcal B$ .  We are considering the unit speed geodesic: $|\dot{c}_Y(t)|=\langle \widetilde Y(c_Y(t)),   \widetilde Y(c_Y(t)) \rangle_g^{1/2}=1$.  Hence,  the arc length measure of the set $I \cap \Omega$  is given by
		\bes
		|I \cap \Omega | =\int_{0}^{l} \chi_{I \cap \Omega}(c_Y(t))~|\dot{c}_Y(t)|~dt = \int_{0}^{l} \chi_{I \cap \Omega}(c_Y(t))~dt.
		\ees
		We claim that: there exists $\mathfrak c>1$ such that $|I \cap \Omega|/|I|\geq \gamma/\mathfrak c$.  Indeed,  using the estimate (\ref{est-B-00}) 
		\bes
	\mu(\mathcal B\cap \Omega ) \leq \mathcal C~l^{d-1}~ |I \cap \Omega| = \mathcal C~l^d \frac{|I\cap \Omega|}{|I|}.
		\ees
Hence,  using the fact that $l\leq 2r$,  the estimate (\ref{est-volume}) yields
		\be \label{est-ball-arc}
		\frac{\mu(\mathcal B\cap \Omega )}{\mu(\mathcal B)} \leq \frac{2^d \mathcal C ~|I\cap \Omega|}{\mathcal C^\prime~|I|}.
		\ee
Let $\mathfrak c= 2^d \mathcal C/\mathcal C^\prime$. Since $\Omega$ is $\gamma$-relatively dense,  we conclude that $|I \cap \Omega | \geq \frac{ \gamma}{\mathfrak c} |I|$.  Let us consider
 \bes
 A = \{t \in [0,1] : g_0 \exp(t l Y) \in I\cap \Omega \}.
 \ees 
 Then $|A|= \frac{|I \cap \Omega|}{l}$.  It follows from Sobolev embedding theorem \cite[Theorem 8.8]{Brz}  that there exists $\mathfrak c_{emb}>0$ such that 
 \be \label{Sobolev-emb}
 \|u\|_{L^\infty(0, 1)} \leq \mathfrak{c}_{emb} \left(\|u\|_{L^p(0,1)} + \|\frac{d}{dt}u\|_{L^p(0,1)}\right), \; \;\; \text{for all $u \in W^{1,p}(0,1)$}.
 \ee
  We now define 
		\bes
		F_Y(t)=\frac{1}{4\|f\|_{L^p(\mathcal B)}\mathfrak{c}_{emb} B^{1/p}~ \sqrt{M_{2d} M_{2d+2}}}~f(g_0 \exp(t l Y)),  \:\: t\in [0, 1].
		\ees 
		Therefore, 
		\beas
		\frac{d}{dt} F_Y(t) &=& \frac{1}{4 \|f\|_{L^p(\mathcal B)}\mathfrak{c}_{emb} B^{1/p}~ \sqrt{M_{2d} M_{2d+2}}}  \frac{d}{du} \Big |_{u=0} f(g_0 \exp (t+u)lY)\ \\
		&=& \frac{1}{4 \|f\|_{L^p(\mathcal B)}\mathfrak{c}_{emb} B^{1/p}~ \sqrt{M_{2d} M_{2d+2}}}  (\widetilde Y f)(g_0~\exp (tl Y)),  
		\eeas
		and by induction
		\be \label{est-derivatives}
		\frac{d^n}{dt^n} F_Y(t)=\frac{l^n}{4\|f\|_{L^p(\mathcal B)}\mathfrak{c}_{emb} B^{1/p}~ \sqrt{M_{2d} M_{2d+2}}} ~(\widetilde Y^n f)(g_0 \exp tlY),  \:\: \textit{  for } \:\: t\in [0, 1]. 
		\ee
		We recall that $l \leq2r\leq 2$. It now follows from the Sobolev estimate \eqref{Sobolev-emb} and the relation (\ref{est-derivatives}) that
		\beas
		\|\frac{d^n}{dt^n} F_Y\|_{L^\infty(0, 1)} &\leq& \mathfrak c_{emb} \left(\|\frac{d^{n}}{dt^{n}} F_Y\|_{L^p(0, 1)} + \|\frac{d^{n+1}}{dt^{n+1}} F_Y\|_{L^p(0, 1)}\right)\\
		&\leq& \frac{2^{n-1}}{\|f\|_{L^p(\mathcal B)} B^{1/p}~ \sqrt{M_{2d} M_{2d+2}}} ~\left(\|\widetilde Y^n f\|_{L^p(\mathcal B)} + \|\widetilde Y^{n+1}f\|_{L^p(\mathcal B)}\right).
		\eeas
		Since $M_{n}$ is log-convex non-decreasing and the fact that $\mathcal B$ is a good ball that, it follows that for all $n\in \N_0$ 
		\beas
		\|\frac{d^n}{dt^n} F_Y\|_{L^\infty(0, 1)} &\leq& \frac{2^n~ B^{(n+1)/p}}{B^{1/p}~ \sqrt{M_{2d} M_{2d+2}}}\sqrt{M_{2d+n} M_{2d+n+2}}\\
		&=& \frac{2^n~B^{n/p}}{\sqrt{M_{2d} M_{2d+2}}}\sqrt{M_{2d+n} M_{2d+n+2}}.
		\eeas
		We now define a new log-convex sequence $\widetilde {\mathcal M} =\{\widetilde M_n\}_{n\in \N_0}$ as follows
		\bea \label{defn-tilde-M}
		\widetilde M_n &=& \frac{2^n~B^{n/p}}{\sqrt{M_{2d} M_{2d+2}}}\sqrt{M_{2d+n} M_{2d+n+2}}.
		\eea
		Then it follows that $F_Y \in C_{\widetilde{\mathcal M}}[0, 1]$.   Clearly,  $\widetilde M_0=1$.  Theorem \ref{thm-NSV} concludes
		\be \label{est-FY-0}  
		\|F_Y\|_{L^\infty[0, 1]} \leq \left(\frac{\Gamma(2n_{F_Y})}{|A|}\right)^{2 n_{F_Y}}  \|F_Y\|_{L^\infty(A)}.
		\ee
		It follows from the condition imposed on the point $g_0$ that 
		\bes
		|F_Y(0)|=  \frac{|f(g_0)|}{4\|f\|_{L^p(\mathcal B)}\mathfrak{c}_{emb} B^{1/p}~ \sqrt{M_{2d} M_{2d+2}}} \geq  \frac{1}{4\mathfrak{c}_{emb} B^{1/p}~ \sqrt{M_{2d} M_{2d+2}}~ \mu(\mathcal B)^{1/p}}.
		\ees  
		We set 
		\be \label{defn-s}
		s= \min \{1,  4^{-1} {\mathfrak{c}_{emb}}^{-1} B^{-1/p} M_{2d}^{-1/2} M_{2d+2}^{-1/2}~ \mu(\mathcal B(e,  r))^{-1/p}\}. 
		\ee

		Since, the measure $\mu$ is $G$-invariant,  we get $\mu(\mathcal B(e, r))= \mu(\mathcal B(g, r))$, for all $g\in G$ and therefore,  the above inequality implies that
		\bes
		\sup_{t\in [0, 1]} |F_Y(t)|\geq s,   \textit{ and hence } \:\: n_{F_Y}\leq n_{\widetilde{\mathcal M},  s}.
		\ees
Consequently,  using the fact that $\Gamma(n)$ is non-decreasing,  it follows from the estimate (\ref{est-FY-0}) that
		\be  \label{est-NSV}
		\|F_Y\|_{L^\infty[0, 1]} \leq \left(\frac{\Gamma(2n_{\widetilde{\mathcal M}, s})}{|A|}\right)^{2 n_{\widetilde{\mathcal M}, s}}  \|F_Y|\|_{L^\infty(A)}.
		\ee
		We recall that $|A|= |I \cap \Omega|/|I|$.  Applying successively the estimates (\ref{est-NSV}) and (\ref{est-ball-arc}) we get
		\beas
		\sup_{\mathcal B\cap \Omega} |f | &\geq&  \sup_{I \cap \Omega} |f |\\
		&=& 4 \|f\|_{L^p(\mathcal B)} \mathfrak{c}_{emb}~ B^{1/p} \sqrt{M_{2d} M_{2d+2}}~ \sup_{t\in A} |F_Y (t)|\\
		&\geq&4 \|f\|_{L^p(\mathcal B)} \mathfrak{c}_{emb}~ B^{1/p} \sqrt{M_{2d} M_{2d+2}}~ \left(\frac{|I \cap \Omega|}{\Gamma(2n_{\widetilde{\mathcal M}, s}) |I|}\right)^{2 n_{\widetilde{\mathcal M}, s}} \sup_{t\in [0, 1]} |F_Y (t)|\\
		&\geq& 4 \|f\|_{L^p(\mathcal B)} \mathfrak{c}_{emb}~ B^{1/p} \sqrt{M_{2d} M_{2d+2}}~ \left(\frac{\mathcal C^\prime \mu(\mathcal B\cap \Omega)}{2^d~\mathcal C ~\Gamma(2n_{\widetilde{\mathcal M}, s}) ~\mu(\mathcal B)}\right)^{2 n_{\widetilde{\mathcal M}, s}} ~ |F_Y (0)|\\
		&\geq& \left(\frac{ \mu(\mathcal B \cap \Omega)}{\mathfrak c ~ \Gamma (2n_{\widetilde{\mathcal M}, s}) ~\mu(\mathcal B)} \right)^{2n_{\widetilde{\mathcal M}, s}}~\mu(\mathcal B)^{-1/p}~\|f\|_{L^p(\mathcal B)},
		\eeas
where $\mathfrak c= 2^d \mathcal C/\mathcal C^\prime>1$ as defined before.  We choose $W\subset \mathcal B$ as the set of points inside the ball where $|f|$ is small relative to its $L^p$-norm on $\mathcal B$ and the measure of $\Omega$.  Precisely,  we define
		\bes
		W = \left\{x \in \mathcal B:  |f(x)|< \left(\frac{\mu(\mathcal B \cap \Omega)}{2 \mathfrak c \Gamma(2n_{\widetilde{\mathcal M}, s}) ~\mu(\mathcal B)}\right)^{2 n_{\widetilde{\mathcal M}, s}} \mu(\mathcal B)^{-1/p}~ \|f\|_{L^p(\mathcal B)}\right \}.
		\ees
		In the following argument,  we assume without loss of generality that $\mu(W)>0$.  Proceeding as before,  we get that
		\be \label{est-W}
		\sup_{W} |f | \geq  \left(\frac{\mu(W)}{\mathfrak c ~ \Gamma(2n_{\widetilde{\mathcal M}, s}) ~\mu(\mathcal B)}\right)^{2 n_{\widetilde{\mathcal M}, s}}~\mu(\mathcal B)^{-1/p}~ \|f \|_{L^p(\mathcal B)} .
		\ee
On the other hand,  the definition of $W$ yields
		\be \label{est-W1}
		\sup_W{|f|}\leq \left(\frac{\mu(\mathcal B \cap \Omega)}{2 \mathfrak c~ \Gamma(2n_{\widetilde{\mathcal M}, s}) ~\mu(\mathcal B)}\right)^{2 n_{\widetilde{\mathcal M}, s}} \mu(\mathcal B)^{-1/p}~ \|f\|_{L^p(\mathcal B)}.
		\ee
		Since $\|f \|_{L^p(\mathcal B(x, r))}>0$,  it follows from the estimates (\ref{est-W}) and (\ref{est-W1}) that $\mu(\mathcal B \cap \Omega) \geq 2 \mu(W)$ and hence,  $\mu(\mathcal B \cap \Omega \backslash W) \geq \frac{\mu(\mathcal B \cap \Omega)}{2}$.  Using the lower bound for $|f|$ on $\mathcal B \cap \Omega\backslash W$ we obtain
		\bea \label{est-final}
		&& \|f\|_{L^p(\mathcal B \cap \Omega)} \geq  \|f\|_{L^p(\mathcal B \cap \Omega\backslash W)} \nonumber\\
		&& \geq \left(\frac{ \mu(\mathcal B \cap \Omega)}{2 \mathfrak c~ \Gamma(2n_{\widetilde{\mathcal M}, s})~ \mu(\mathcal B)}\right)^{2 n_{\widetilde{\mathcal M}, s}} ~ \mu(\mathcal B)^{-1/p}~\|f \|_{L^p(\mathcal B)}~\mu(\mathcal B \cap \Omega \backslash W)^{1/p} \nonumber\\
		&& \geq \left(\frac{\mu(\mathcal B \cap \Omega)}{2 \mathfrak c~ \Gamma(2n_{\widetilde{\mathcal M}, s}) ~\mu(\mathcal B)}\right)^{2 n_{\widetilde{\mathcal M}, s}+\frac{1}{p}} \|f \|_{L^p(\mathcal B)} \nonumber\\
		&& \geq \left(\frac{  \gamma}{2\mathfrak c~\Gamma(2n_{\widetilde{\mathcal M}, s})}\right)^{2 n_{\widetilde{\mathcal M}, s}+\frac{1}{p}} \|f \|_{L^p(\mathcal B)}.
		\eea
We notice that the above estimate holds as well when $\|f\|_{L^p(\mathcal B_j)} = 0$.  Using the estimate (\ref{est-good}) it follows from (\ref{est-final}) that
		\beas
		\int_{G} |f(g)|^p~d\mu(g) &\leq& 2 \sum_{\textit{ good balls }} \int_{\mathcal B_j} |f(g)|^p~d\mu(g)\\
		&\leq& 2 \sum_{\textit{ good balls }} \left(\frac{2\mathfrak c \Gamma(2n_{\widetilde{\mathcal M}, s})}{ \gamma}\right)^{2 p n_{\widetilde{\mathcal M}, s}+1} \|f \|_{L^p(\mathcal B_j \cap \Omega)}^p\\
		&\leq& 2 N  \left(\frac{2 \mathfrak c\Gamma(2n_{\widetilde{\mathcal M}, s})}{ \gamma}\right)^{2 p n_{\widetilde{\mathcal M}, s}+1} \|f \|_{L^p(\mathcal B_{good} \cap \Omega)}^p\\
		&\leq&2N \left(\frac{2 \mathfrak c \Gamma(2n_{\widetilde{\mathcal M}, s})}{ \gamma}\right)^{2 p n_{\widetilde{\mathcal M}, s}+1} \|f \|_{L^p(\Omega)}^p,
		\eeas
 This completes the proof for $p\in [1, \infty)$.  For the case $p=\infty$, the proof follows similarly.  However,  in this case a ball $\mathcal B_j$ is good if
		 for all $n\in \N$
		\be \label{defn-good-infty}
		\| \mathcal L^{n} f\|_{L^\infty(\mathcal B_j)} \leq B^{n} M_{2n} \| f\|_{L^\infty(\mathcal B_j)}.
		\ee
		Let $g^\prime\in G$ such that $f(g^\prime)=\|f\|_{L^\infty(G)}$ and assume $g^\prime\in \mathcal B_j$ for some $j$.  As in the previous argument,  using the definition of bad balls,  we show that $\mathcal B_j$ must be a good ball. Therefore,  $\|f\|_{L^\infty(\mathcal B_{good})}= \|f\|_{L^\infty(G)}$.  We now proceed as before, considering each good ball $\mathcal B_j$.  Precisely,  we choose $g_0\in \mathcal B_j$ such that $|f(g_0)| \geq \frac{1}{2}\|f\|_{L^\infty(\mathcal B_j)}$.  Defining  $F_Y$ as in the earlier step and following the same reasoning,  we get $C>0$ such that
		\bes
		\|f\|_{L^\infty(\mathcal B_j \cap \Omega)}\geq \left(\frac{ C \gamma}{\Gamma(2n_{\widetilde{\mathcal M}, s})}\right)^{2 n_{\widetilde{\mathcal M}, s}} \|f \|_{L^\infty(\mathcal B_j)}.
		\ees 
		Summing over all good balls completes the proof.
	\end{proof}
	
	\begin{proof}[Proof of Corollary \ref{thm-main-2}]
		Let $\gamma \in (0, 1]$ and $C_W>0$. We show that there exists $C>0$ depending on $C_W, \gamma,  G$ such that evenever a function $f \in L^2(G)$ satisfies
		\be \label{est-proof}
		\sum_{[\xi] \in \widehat G} d_\xi ~W(|\xi|)^2 ~ |\widehat{f}(\xi)|^2  \leq C_W^2 \|f\|_{L^2(G)}^2,
		\ee
		and $\Omega$ is a $\gamma$-relatively dense set in $G$, we get $\|f\|_{L^2(G)} \leq C \|f\|_{L^2(\Omega)}$.
		In view of Theorem \ref{thm-main}, it is enough to show that under the hypothesis (i), (ii) and (iii) on $W$ and the estimate (\ref{est-proof}) implies that $f \in C_{\mathcal M}(G)$, for some quasi-analytic sequence $\mathcal M$.  Without loss of generality,  we assume $\|f\|_{L^2(G)}=1$ and $W \equiv 1$ on $[0, 1]$.  Let us define 
		\bes
		M_n=\sup_{\lambda \geq 0} \frac{\lambda^n}{W(\lambda)}=\sup_{\lambda \geq 1} \frac{\lambda^n}{W(\lambda)}.
		\ees
		We observe that $M_0=1$ and $M_n$ is an increasing log-convex sequence.  It is shown  in \cite[Prop. 2.2]{JM} that (\ref{decay-weight}) implies $\{M_n\}_{n=0}^\infty$ is a quasi-analytic sequence.  Using Plancherel formula (\ref{Plancherel}) and hypothesis (\ref{est-proof}),  we get
		\beas 
		\|\mathcal L^{n} f\|_{L^2(G)}&=&  \left(\sum_{[\xi] \in \widehat G} d_\xi ~|\xi|^{4n}~ \|\widehat f(\xi)\|_{HS}^2 \right)^{1/2}  \\
		&\leq& \sup_{[\xi] \widehat G} \frac{|\xi|^{2n}}{W(|\xi|)}  \left(\sum_{[\xi] \in \widehat G} d_\xi~W(|\xi |)^2~ \|\widehat f(\xi)\|_{HS}^2 \right)^{1/2} \\
		&\leq& \sup_{\lambda\geq 0} \frac{\lambda^{2n}}{W(\lambda)}~ C_W = C_W~M_{2n}.
		\eeas
		Clearly, $\mathcal M^\prime= \{C_W M_n\}_{n=0}^\infty$ is a quasi-analytic sequence and $f\in C_{\mathcal M^\prime}(G)$.  Theorem {\ref{thm-main-2}} follows now from Theorem \ref{thm-main}.
		
	\end{proof}
	
	\begin{proof}[Proof of Corollary \ref{thm-spectral}]
		Let $f\in L^2(G)$.  It follows from the definition (\ref{defn-PLambda}) that
		\bes 
		\mathcal L^n P_\Lambda f(g)=  \sum_{\{[\xi]\in \widehat G: |\xi|\leq \sqrt{\Lambda}\}} d_\xi ~ |\xi|^{2n}~Tr\left(\xi(g) \widehat f(\xi)\right),
		\ees
		and hence,  by Plancherel formula (\ref{Plancherel}) 
		\beas
		\|\mathcal L^{n} (P_\Lambda f)\|_{L^2(G)} &=& \left( \sum_{\{[\xi] \in \widehat G: |\xi|\leq \sqrt{\Lambda}\}} d_\xi ~|\xi|^{4n}~ \|\widehat f(\xi)\|_{HS}^2 \right)^{1/2} \leq \Lambda^{n}~\|P_\Lambda f\|_{L^2(G)}.
		\eeas
		It is enough to prove the required estimate for all $\|P_\lambda f\|_{L^2(G)}=1$.  Let $\mathcal M=\{\Lambda^{n/2}\}_{n=0}^\infty$.  Clearly,  the sequence $\mathcal M$ is quasi-analytic.  Also,  $P_\Lambda f\in C_{\mathcal M}(G)$. Therefore,  by Theorem \ref{thm-main}
		it follows that 
		\bes
		\|f\|_{L^2(G)}\leq  \left( \frac{C \Gamma(2n_{\widetilde{\mathcal M}, s})}{\gamma} \right)^{2 n_{\widetilde{\mathcal M}, s}+\frac{1}{2}}   \|f\|_{L^2(\Omega)},
		\ees
		for all $\gamma$-relatively dense set $\Omega$. It now remains to verify the constants in this particular class.  Clearly,  for $\mathcal M=\{\Lambda^{n/2}\}_{n=0}^\infty$, we obtain $\widetilde{\mathcal M}=\{c^n\Lambda^{n/2}\}$,  where $c^n=2^nr^n  \mathfrak C^n ~ B^n$. Also, $n_{\widetilde{\mathcal M}, s}\leq e \sqrt{\Lambda}$ and $\Gamma(n)=4e^4$ for all $n$.  Therefore the constant appearing in the right-hand side of the above inequality reduces to $(C/\gamma)^{2e\sqrt{\Lambda}+1/2}$.
	\end{proof}
	
	\section{Uniqueness results on quotient spaces}
	
	In this section,  we prove the uniqueness result (Theorem \ref{thm-X}) in the context of homogeneous spaces $G/H$,  where $G$ is a compact, connected Lie group and $H$ is any closed subgroup of $G$.  This gives, in particular,  a generalized Logvinenko-Sereda type result on the $d$-dimensional sphere $S^d=SO(d+1)/SO(d)$.   We will denote the base point in $G/H$ by $o =eH$. Conceptually,  the proof of the result is similar to that of Theorem \ref{thm-main}, with most of the ideas being contained in there.

	We recall that $\mathfrak h$ denotes the Lie algebra of $H$,  and we take $\mathfrak q$ to be the orthogonal complement of $\mathfrak h$ in $\mathfrak g$ with respect to the $Ad$-invariant inner product $\langle,  \rangle$ on $\mathfrak g$.  The Lie algebra $\mathfrak q$ can be identified with the tangent space $T_o(G/H)$ at the origin $o=eH$.   Let $l=\dim \mathfrak q$.  Since $\mathcal L$ is both left- and right-invariant,  it naturally descends to a left-invariant differential operator $\widetilde{\mathcal L}$ on $G/H$.   Let $\pi: G \ra G/H$ be the canonical projection map.  For a function $f$ defined on $G/H$,  we denote its pullback to $G$ by $f^\prime =f \circ \pi$.  Clearly,  $f^\prime$ is a right $H$-invariant function on $G$.   It can be shown that for  $f\in C^\infty(G/H)$,  $\widetilde{\mathcal L} f= \mathcal L f^\prime$. 
	As we have mentioned earlier,  every element $D\in \mathcal U(\mathfrak q)$ corresponds to a differential operator $d_r(D)$ on $G/H$ defined as in (\ref{defn-dr}).  In fact,   these are right-invariant differential operators for functions $f^\prime$ on $G$; hence,  they descend to a differential operator for functions on $G/H$.  
	
	Let $Exp: q \ra G/H$ be the exponential map given by $Exp(X) = \exp(X) H$.  By identification of $q$ with the tangent space $T_o(G/H)$,  we
	thus identify $Exp$ with the exponential map at $o$ associated with the Riemannian connection.  Since $G/H$ is a compact Riemannian manifold,  the injectivity radius $\widetilde r_{inj}$ of $G/H$ is positive.  Let $B(0, r)\subset q$ be the ball of radius $r$.  Then,  for all $0< r< \widetilde r_{inj}$,  the Riemannian exponential map from $B(0, r) \ra \widetilde{\mathcal B}(gH,  r)$ is surjective for all $gH\in G/H$.   Precisely,  for any $gH\in G/H$,  the map $X\ra g\exp(X)H$ is a surjective map between these balls.  As in the case of full group $G$,  there exists $\mathfrak r\in (0,  \widetilde r_{inj})$ such that for all $0<r< \mathfrak r$,  the volume form $\widetilde{\mu}$ in geodesic polar coordinates around $gH$ has the following property: there exists $\eta_1,\eta_2>0$ such that
	\be \label{vol-est-GH}
	\eta_2 r^{l-1}~dr~d\sigma(X)\leq d\widetilde{\mu}(g Exp(rX)) \leq \eta_1 r^{l-1}~dr~d\sigma(X), 
	\ee
	where $d\sigma(X)$ denotes the canonical normalized measure of the unit sphere $\mathcal S^{l-1}$ in $\mathfrak q$.  We now conclude the proof of Theorem \ref{thm-X} by following the argument of Theorem \ref{thm-main} verbatim.
	
	\begin{proof}[Proof of Theorem \ref{thm-X}]
		Without loss of generality we assume $\|f\|_{L^p(G/H)}=1$.  Let $r \in (0,  \mathfrak r/2)$ be such that $|\widetilde{\mathcal B}(gH, r) \cap \widetilde \Omega| \geq \gamma |\widetilde {\mathcal B}(gH, r)|$.  We consider $p\in[1, \infty)$ and without loss of generality we assume $r\leq 1$.   We recall $f^\prime =f\circ \pi \in C^\infty(G)$. We cover $G/H$ by finite number of balls with radius $r$ and divide them in good and bad balls.  Precisely,  a ball $\widetilde{\mathcal B}$ is good,  if for all $n\in \N$ and $X \in \mathcal S^{l-1} \subset \mathfrak q$
		\be \label{defn-good-GH}
		\int_{\pi^{-1}(\widetilde{\mathcal B})} \left|\widetilde{X}^{n} f^\prime(g)\right|^p ~d\mu(g) \leq B^{n} \bar M_{n}^p \int_{\pi^{-1}(\widetilde{\mathcal B})} |f^\prime(g)|^p~d\mu(g),
		\ee
		It follows from the definition that if $\widetilde{\mathcal B}$ is bad,  then 
		\be \label{est-bad}
		\int_{\pi^{-1}(\widetilde{\mathcal B})}  |f^\prime(g)|^p~d\mu(g) < \sum_{n\in \N}\frac{1}{B^{n}~\bar M_{n}^p}  \int_{\pi^{-1}(\widetilde{\mathcal B})}  |\widetilde X^{n} f^\prime(g)|^p~d\mu(g), \textit{ for some } X\in S^{l-1}.
		\ee
		Since $f\in C_{\mathcal M, p}(G/H)$, and $d\widetilde \mu$ is a normalised measure on $G$, we get that $\|\widetilde{\mathcal L}^n f\|_{L^1(G/H)}\leq M_{2n}$.    Using the fact that $\mathcal L f^\prime=\widetilde{\mathcal L}f$ , it follows  that  $\|\mathcal L^n f^\prime \|_{L^1(G)}\leq M_{2n}$.  Consequently,  using Lemma \ref{lem-Sobolev}, we obtain that for all $X\in \mathcal S^{l-1}$
		\be \label{est-bad-union}
		\int_{\pi^{-1}(\widetilde{\mathcal B})}  |\widetilde X^{n} f^\prime(g)|^p~d\mu(g)\leq \widetilde{\mu}(\widetilde{\mathcal B})~ \|\widetilde X^{n} f^\prime\|_{L^\infty(\pi^{-1}(\widetilde{\mathcal B}))}^p\leq  \mathfrak C^{p(n+1)} ~ \bar M_n^p,
		\ee
		where the sequence $\{\bar M_n\}$ is defined in \eqref{defn-barM}. Following the proof of Theorem \ref{thm-main},  we need to work with only a fixed good ball.  Specifically,  it is enough to show that there exists $C>0$ such that for any good ball $\widetilde{\mathcal B}=\widetilde {\mathcal B}(gH, r)$
		\be \label{claim-X}
		\|f\|_{L^p(\widetilde{\mathcal B} \cap \widetilde\Omega)}\geq \left(\frac{ C \gamma}{\Gamma(2n_{\widetilde{\mathcal M}, s})}\right)^{2 n_{\widetilde{\mathcal M}, s}+\frac{1}{p}} \|f \|_{L^p(\widetilde{\mathcal B})}.
		\ee
To prove this,  let $g_0H \in\widetilde{\mathcal B}$ be a point such that $|f(g_0H)| \geq \|f\|_{L^p(\widetilde{\mathcal B})} ~|\widetilde{\mathcal B}|^{-1/p}$.  For $X\in \mathcal S^{l-1}\subset q$,  we take the geodesic curve segment $c_X(t)= g_0 Exp (tX) \subset \widetilde{\mathcal B}$ starting at $g_0H$.   Applying estimate (\ref{vol-est-GH}),  we get  $Y \in \mathcal S^{l-1}$ such that
		\be \label{est-B-0}
		|\widetilde{\mathcal B}\cap \widetilde{\Omega} | \leq \eta_1 \int_{0}^{2r} \chi_{\widetilde{\mathcal B}\cap \widetilde{\Omega}}(c_{Y}(t))~t^{l-1}~ dt.
		\ee  
Let $\kappa=\sup\{t\in[0,2r]: c_Y(t)\in \widetilde{\mathcal B}\}$ and $I=\{c_Y(t): 0\leq t < \kappa\} \subset \widetilde{\mathcal B}$ . 	Then,  $\gamma$-relatively denseness of $\widetilde\Omega$ and $(\ref{est-B-0})$ imply that there exists $\eta >1$ such that,  the arc length measure of the set $I\cap \widetilde \Omega$ satisfies $|I \cap \widetilde\Omega|/|I|\geq \gamma/\eta$.  Let $A := \{t \in [0,1] : g_0 Exp(t \kappa Y) \in I\cap \widetilde\Omega \}$. Then $|A|= \frac{|I\cap \widetilde\Omega|}{|I|}$.  Let us define 
		\bes
		F_Y(t)=\frac{1}{4\|f^\prime\|_{L^p(\pi^{-}{(\widetilde{\mathcal B}}))}\mathfrak{c}_{emb} B^{1/p}~ \sqrt{M_{2d} M_{2d+2}}}~f^\prime (g_0 \exp(t \kappa Y)),  \:\: t\in [0, 1].
		\ees 
		It follows that for $n\in \N_0$
		\be \label{est-derivatives-GH}
		\frac{d^n}{dt^n} F_Y(t)=\frac{\kappa^n}{4 \|f^\prime\|_{L^p(\pi^{-}{(\widetilde{\mathcal B}}))} \mathfrak{c}_{emb} B^{1/p}~ \sqrt{M_{2d} M_{2d+2}}} ~(\widetilde Y^n f^\prime)(g_0 \exp (t \kappa Y)),  \:\: \textit{  for } \:\: t\in [0, 1]. 
		\ee
We recall that $\kappa \leq2r\leq 2$. Using Sobolev estimate \eqref{Sobolev-emb}, it follows from above relation that
		\beas
		\|\frac{d^n}{dt^n} F_Y\|_{L^\infty(0, 1)} &\leq& \mathfrak c_{emb} \left(\|\frac{d^{n}}{dt^{n}} F_Y\|_{L^p(0, 1)} + \|\frac{d^{n+1}}{dt^{n+1}} F_Y\|_{L^p(0, 1)}\right)\\
		&\leq& \frac{2^{n-1}}{\|f^\prime\|_{L^p(\pi^{-}{(\widetilde{\mathcal B}}))} B^{1/p}~ \sqrt{M_{2d} M_{2d+2}}} ~\left(\|\widetilde Y^n f^\prime \|_{L^p(\pi^{-1}(\widetilde{\mathcal B}))} + \|\widetilde Y^{n+1}f^\prime \|_{L^p(\pi^{-1}(\widetilde{\mathcal B}))}\right).
		\eeas
		Since $M_{n}$ is log-convex non-decreasing and the fact that $\widetilde{\mathcal B}$ is a good ball that, it follows that for all $n\in \N_0$ 
		\beas
		\|\frac{d^n}{dt^n} F_Y\|_{L^\infty(0, 1)} &\leq& \frac{2^n~ B^{(n+1)/p}}{B^{1/p}~ \sqrt{M_{2d} M_{2d+2}}}\sqrt{M_{2d+n} M_{2d+n+2}}\\
		&=& \frac{2^n~B^{n/p}}{\sqrt{M_{2d} M_{2d+2}}}\sqrt{M_{2d+n} M_{2d+n+2}}.
		\eeas
		We consider the sequence $\widetilde{\mathcal M}=\{\widetilde{M}_n\}_{n\in \N_0}$ as defined in (\ref{defn-tilde-M}) and conclude from above that $F_Y \in C_{\widetilde{\mathcal M}}[0, 1]$.  The rest of the argument to show (\ref{claim-X}) now follows as before.
	\end{proof}
	
	\section{Unique continuation property on $G$ and $G/H$}
	We begin the section with the proof of Theorem \ref{thm-cher-G}.
	\begin{proof}[Proof of Theorem \ref{thm-cher-G}]
		We note that it is enough to prove the theorem for $p=1$.  Using the left-invariance of $\widetilde D$ for $D \in \mathcal U(\mathfrak g)$ and that of the Haar measure $\mu$, we may assume without loss of generality that $f \in \widetilde C_{\mathcal M,1}(G)$ vanishes along with $\widetilde D f$ for all $D \in \mathcal U(\mathfrak g)$ at the identity $e$.
		
		For each $X\in \mathcal S^{d-1}\subset \mathfrak g$,  we define the function
		\bes
		F_X(t)= f(\exp tX), \:\: t\in \R.
		\ees  
		In view of Lemma \ref{lem-Sobolev} and the log-convexity of $\{M_n\}$,  it follows that 
		\bes
		\|F_X^{(n)}\|_{L^\infty(\R)}\leq  B_f \beta_f^{d+n/2}\mathfrak C^{n+1} \sqrt{M_{2d+n-1} M_{2d +n+1}}, \:\: \textit{ for all } \:\: n\in \N_0.
		\ees
		Clearly,  the sequence $\{\overline{M_n}\}$,  where $\overline{M_n}=\sqrt{M_{2d+n-1} M_{2d +n+1}}$,  is quasi-analytic.  Moreover,  $F_X^{(n)}(0)=0$ for all $n$.  Hence,  by the Denjoy-Carleman theorem (Theorem \ref{thm-DC}),  we get $f(\exp(tX))=0$, for all $t \in \R$.  Since this is true for all $X \in S^{d-1} $, we conclude that $f(\exp Y)=0$ for all $Y\in \mathfrak g$.  The surjectivity of the exponential map concludes that $f$ vanishes identically on $G$. 
	\end{proof}
	
	To prove the uniqueness theorem on the quotient space,  we need to deal with right-invariant differential operators.  In this direction,  we need the following analogue of Lemma \ref{lem-Sobolev} for functions on $G/H$. 
	\begin{lem} \label{lem-Sobolev-GH}
		There exists $\mathfrak C>1$ such that for all $f\in C^\infty(G/H)$,  $X\in \mathcal S^{l-1}\subset \mathfrak q$ and for all $n\in \N_0$,  we have
		\beas
		\| d_r(X^n) f \|_{L^\infty(G/H)} &\leq &  \mathfrak C^{n+1} 
		\begin{cases} 
			\|\widetilde {\mathcal L}^{d+n/2} f\|_{L^1(G/H)},  & \text{ if $n$ is even }; \\
			\|\widetilde{\mathcal L}^{d+(n-1)/2} f\|_{L^1(G/H)}^{1/2}~\|\widetilde{\mathcal L}^{d+(n+1)/2} f\|_{L^1(G/H)}^{1/2},  & \text{ if $n$ is odd }.
		\end{cases}
		\eeas
	\end{lem} 
	\begin{proof}
		Let $f\in C^\infty(G/H)$ and $X\in \mathcal S^{l-1}\subset \mathfrak q$.  We now apply Lemma \ref{lem-Sobolev} to the function $f^\prime$.  However, in this case, we must replace the left-invariant differential operator $\widetilde X$ by the right-invariant one $d_r(X^n)$.  To account for this change,  we use the relation (\ref{rel-D-r}) instead of (\ref{rel-D}).  Repeating the same argument then yields the required estimate for the right $H$-invariant function $f^\prime$.  Moreover,  since the operators $d_r(X^n)$ are right-invariant,  we have  $d_r(X^n) f^\prime= d_r(X^n) f$.  This completes the proof. 
	\end{proof}

	\begin{proof}[Proof of Theorem \ref{thm-cher-GH}]
		The proof is similar to that of Theorem \ref{thm-cher-G}.  As before, using the compactness of $G/H$,  it is enough to prove the theorem for $p=1$.   By the hypothesis $d_r(D)f(o)=0$,  for all $D\in \mathcal U(\mathfrak q)$.  For each $X\in \mathcal S^{l-1} \subset \mathfrak q$,  we define the function
		\bes
		\mathcal F_X(t)= f\left(\exp (tX)H\right),  \:\: t\in \R.
		\ees  
		Then 
		\bes
		\frac{d^n}{dt^n} \mathcal F_X(t)= d_r(X^n) f(\exp(tX)H)
		\ees
		In view of Lemma \ref{lem-Sobolev-GH} and the log-convexity of $\{M_n\}$,  it follows that 
		\bes
		\|\mathcal F_X^{(n)}\|_{L^\infty(\R)}\leq B_f \beta_f^{d+n/2} \mathfrak C^{n+1} \sqrt{M_{2d+n-1} M_{2d +n+1}}, \:\: \textit{ for all } \:\: n\in \N_0.
		\ees 
		This shows that $\mathcal F_X\in C_{\overline{\mathcal M}}(\R)$.  Clearly,   $\mathcal F_X^{(n)}(0)=0$,  for all $n\in \N$. Hence,  by the Denjoy-Carleman theorem,  we conclude that $f(\exp(tX)H)=0$, for all $t\in \R$.  Therefore,  $f(\exp(X)H)=0$ for all $X\in \mathfrak q$.  Since the map $Exp: \mathfrak q \ra G/H$ is surjective,  it follows that $f(gH)=0$ for all $g\in G$. This completes the proof. 
		
\end{proof}
	
	
	

\end{document}